\documentclass[letterpaper, 10pt,reqno]{amsart}

\usepackage[usenames,dvipsnames]{xcolor}

\usepackage[portrait,margin=3cm]{geometry}
\usepackage{mathrsfs}

\usepackage{amssymb,amsthm,amsfonts,amsbsy,amstext,latexsym}
\usepackage[reqno]{amsmath}

\usepackage{color}
\usepackage{graphicx}
\usepackage{bbm}
\usepackage{comment}
\usepackage{a4wide,graphicx, listings,lscape, epstopdf, units, textcomp, fancyhdr, url, soul, color}
\usepackage{tikz}
\usetikzlibrary{shapes.geometric}
\usepackage[textsize=small]{todonotes}
\usepackage{enumitem}

\usepackage{mathtools}
\mathtoolsset{showonlyrefs}

\definecolor{MyDarkblue}{rgb}{0,0.08,0.50}
\definecolor{Brickred}{rgb}{0.65,0.08,0}

\usepackage{hyperref,upref}
\hypersetup{
	colorlinks=true,       
	linkcolor=MyDarkblue,          
	citecolor=Brickred,        
	filecolor=red,      
	urlcolor=cyan           
}

\newtheorem*{theorem*}{Theorem}
\newtheorem{theorem}{Theorem}[section]
\newtheorem{lemma}[theorem]{Lemma}
\newtheorem{proposition}[theorem]{Proposition}
\newtheorem{corollary}[theorem]{Corollary}

\theoremstyle{definition}
\newtheorem{definition}[theorem]{Definition}
\newtheorem{assumption}[theorem]{Assumption}
\newtheorem{remark}[theorem]{Remark}

\renewcommand{\P}{\mathbb{P}}

\newcommand{\Pv}{\mathbb{P}}

\newcommand{\eps}{\varepsilon}

\newcommand{\cB}{\mathcal{B}}

\newcommand{\cL}{\mathcal{L}}
\newcommand{\cM}{\mathcal{M}}\newcommand{\cO}{\mathcal{O}}
\newcommand{\cP}{\mathcal{P}}
\newcommand{\cT}{\mathcal{T}}\newcommand{\cU}{\mathcal{U}}

\newcommand{\e}{{\mathrm e}}

\newcommand{\R}{\mathbb{R}}
\newcommand{\N}{\mathbb{N}}
\newcommand{\Z}{\mathbb{Z}}

\newcommand{\dd}{\mathrm{d}}

\renewcommand{\emptyset}{\varnothing}

\newcommand*{\wt}{\widetilde}

\newcommand*{\be}{\begin{equation}}
	\newcommand*{\ee}{\end{equation}}
\newcommand*{\ba}{\begin{aligned}}
	\newcommand*{\ea}{\end{aligned}}
\newcommand*{\barr}{\begin{array}{c}}
	\newcommand*{\earr}{\end{array}}
\def \toinp    {\buildrel {\Pv}\over{\longrightarrow}}
\def \toindis  {\buildrel {d}\over{\longrightarrow}}
\def \toas     {\buildrel {a.s.}\over{\longrightarrow}}

\makeatletter
\def\namedlabel#1#2{\begingroup
	#2%
	\def\@currentlabel{#2}%
	\phantomsection\label{#1}\endgroup
}

\makeatother

\newcommand{\bes}{\begin{equation*}}
	\newcommand{\ees}{\end{equation*}}

\renewcommand{\P}[1]{\mathbb{P}\!\left(#1\right)}
\newcommand{\E}[1]{\mathbb{E}\left[#1\right]}

\renewcommand{\N}{\mathbb{N}}

\numberwithin{equation}{section}

\renewcommand{\e}{\mathrm{e}}

\DeclareMathOperator{\exponentialrv}{Exp}
\newcommand{\Exp}[1]{\exponentialrv\left( #1 \right)}

\DeclareMathOperator{\Deg}{deg}
\newcommand{\outdeg}[1]{\ensuremath{\Deg^{+}(#1)}}

\newcommand{\ensymboldremark}{\hfill$\blacktriangleleft$}

\newcommand{\invisible}[1]{}

\title[A star is born: Explosive CMJ branching processes]{A star is born: Explosive Crump--Mode--Jagers branching processes}
\author[Lodewijks]{Bas Lodewijks}
\address{School of Mathematics and Physical Sciences, University of Sheffield}
\email{bas.lodewijks@sheffield.ac.uk}

\begin{document}
	\maketitle
	\noindent  \bigskip
	\\
	
	\begin{abstract}
		We study a family of Crump--Mode--Jagers branching processes in random environment that explode, i.e.\ that grow infinitely large in finite time with positive probability. Building on recent work of the author and Iyer (``On the structure of genealogical trees associated with explosive Crump--Mode--Jagers branching processes", arXiv:2311.14664, 2023), we weaken certain assumptions required to prove that the branching process, at the time of explosion, contains a (unique) individual with infinite offspring. We then apply these results to \emph{super-linear preferential attachment models}. In particular, we fill gaps in some of the cases analysed in Appendix A of the work of the author and Iyer and study a large range of previously unattainable cases.
	\end{abstract} 
	
\section{Introduction}\label{sec:intro}

In a Crump--Mode--Jagers (CMJ) branching process (named after~\cite{crump-mode, jagers-origin}) an ancestral \emph{root} individual produces offspring according to a collection of points on the non-negative real line. Each individual `born' produces offspring according to an identically distributed collection of points, translated by their birth time. One is generally interested in properties of the population as a function of time. Classical work from the 1970s and '80s related to this model generally deals with the \emph{Malthusian case}, which, informally, refers to the fact that the population grows exponentially in time; see e.g.~\cite{Arthreya-ney, jagers-book} and the references therein for an overview of classical CMJ theory. 

Far fewer results exist for CMJ branching processes when a Malthusian parameter does not exist. One family of CMJ branching processes for which this is the case, is the family of \emph{explosive} CMJ branching processes. Here, the total progeny of the branching process grows infinitely large in \emph{finite} time with positive probability. Early work on such CMJ branching processes by Sevast'yanov~\cite{Sev67,Sev70}, Grey~\cite{Grey74}, Grishechkin~\cite{Gris87}, and Vatutin~\cite{Vat87} studies necessary and sufficient conditions for (non-)explosion of Bellman-Harris processes (a special type  CMJ branching processes). More recent work on this topic distinguishes between processes where individuals produce finitely many offspring in finite time almost surely, and processes where individuals can produce infinite offspring themselves in finite time. In the former case, Komj\'athy provides general criteria for explosion of CMJ branching processes in terms of solutions of a functional fixed point equation~\cite{Komjathy2016ExplosiveCB}, and also extends the necessary and sufficient criteria for explosion in branching random walks in~\cite{amini-et-al}. In the latter case, Oliveira and Spencer, and very recently Iyer~\cite{Iyer24}, study a particular family of explosive CMJ branching processes with exponential inter-birth times~\cite{Oliveira-spencer}, and Sagitov considers Bellman-Harris processes where individuals, upon death, may produce infinite offspring~\cite{Sag16,Sag17}.

Recent work of the author and Iyer investigates a family of explosive CMJ branching processes in \emph{random environment}~\cite{IyerLod23}. Here, each individual born in the branching process is assigned an i.i.d.\ random weight, and the distribution of the point process that governs the offspring of the individual depends on their weight. Sufficient conditions for the emergence of \emph{infinite stars} (individuals that produce infinite offspring) and \emph{infinite paths} (an infinite lineage of descendants) in the branching process, stopped at the (finite) time of explosion are formulated in a general set-up. Other sufficient criteria for explosion for CMJ processes in random environment are also studied in~\cite{limiting-structure}. 

\textbf{Overview of our contribution. } In this paper we study the emergence of \emph{infinite stars} in a explosive CMJ processes in random environment. In particular, we weaken certain assumptions provided in previous work of the author and Iyer~\cite{IyerLod23}, under which we prove that an infinite star appears almost surely.

We use these results in an application to \emph{super-linear preferential attachment} trees with fitness. This model consists of a sequence of discrete-time trees in which vertices are assigned vertex-weights and arrive one by one. A new vertex connects to a random vertex in the tree, selected with a preference for vertices with high degree and/or vertex-weight. We largely extend the class of models for which such results can be proved. This significantly improves on the results in~\cite{IyerLod23}, and also adds to the recent work of Iyer~\cite{Iyer24} on persistence in (super-linear) preferential attachment trees (without fitness).

\textbf{Model definition. } We consider \emph{individuals} in the process as being labelled by elements of the infinite \emph{Ulam-Harris} tree $\cU_\infty : = \bigcup_{n \geq 0} \N^{n}$; where $\N^{0} := \{\varnothing\}$ contains a single element which we call  the \emph{root}. 
We denote elements $u \in \cU_{\infty}$ as tuples, so that, if $u = (u_{1}, \ldots, u_{k}) \in \N^{k}, k \geq 1$, we write $u = u_{1} \cdots u_{k}$. An individual $u = u_1u_2\cdots u_k$ is to be interpreted recursively as the $u_k^{\mathrm{th}}$ child of the individual $u_1 \cdots u_{k-1}$; for example, $1, 2, \ldots$ represent the offspring of $\varnothing$. Further, for individuals $u=(u_1,\ldots, u_k)$ and $v=(v_1,\ldots, v_\ell)$, we let $uv=(u_1,\ldots, u_k,v_1,\ldots, v_\ell)$ denote their concatenation. Suppose that $(\Omega, \Sigma, \mathbb{P})$ is a complete probability space and $(S, \mathcal{S})$ is a measure space. We also equip $\mathcal{U}_{\infty}$ with the sigma algebra generated by singleton sets. Then, we fix random mappings $X: \Omega \times \cU_{\infty} \rightarrow [0, \infty]$, $W: \Omega \times \cU_{\infty} \rightarrow S$, and define $(X, W): \Omega \times \cU_{\infty} \rightarrow [0, \infty] \times S$, so that $(\omega,u) \mapsto ((X(u))(\omega), W_{u}(\omega))$. 
In general, for $u \in \cU_{\infty}$ and $j\in \N$, one interprets $W_{u}$ as a `weight' associated with $u$, and $X(uj)$ the waiting time or \emph{inter-birth} times between the birth of the $(j-1)$th and $j$th \emph{child} of $u$.

We use the values of $X$ to associate \emph{birth times} $\mathcal{B}(u)$ to individuals $u \in \cU_{\infty}$. In particular, we define $\mathcal{B}: \Omega \times \cU_{\infty} \rightarrow [0, \infty]$ recursively as follows:  
\[\mathcal{B}(\varnothing) : = 0 \quad \text{and, for $u \in \cU_{\infty}, i \in \mathbb{N}$,} \quad \mathcal{B}(ui) := \mathcal{B}(u) + \sum_{j=1}^{i} X(uj).\]
We note that a value of $X(ui) = \infty$ indicates that the individual $u$ has stopped producing offspring, and produces at most $i-1$ children. 

We introduce some notation related to elements $u \in \cU_{\infty}$. We use $|\cdot|$ to measure the \emph{length} of a tuple $u$, so that $|u|=0$ when $u=\varnothing$ and $|u|=k$ when $u = u_{1} \cdots u_{k}$. If, for some $x \in \cU_{\infty}$, we have $x = u v$, we say $u$ is a \emph{ancestor} of $x$. Further, given $\ell \leq |u|$, we write $u_{|_\ell} := u_{1} \cdots u_{\ell}$ for the different ancestors of $u$.  We equip $\cU_{\infty}$ with the lexicographic total order $\leq_{L}$. Given elements $u, v$, we say $u \leq_{L} v$ if either $u$ is an ancestor of $v$, or $u_{\ell} < v_{\ell}$ where $\ell = \min \left\{i \in \mathbb{N}: u_{i} \neq v_{i} \right\}$. For $u \in \cU_{\infty}$, we let $\mathcal{P}_{i}(u)$ denote the time, after the \emph{birth} of $u$, required for $u$ to produce $i$ offspring. That is,
\be\label{eq:P}
\mathcal{P}_{i}(u) := \sum_{j=1}^{i} X(uj) \quad \text{and} \quad \mathcal{P}(u) := \mathcal{P}_{\infty}(u) = \sum_{j=1}^{\infty} X(uj).
\ee	
We use the notation $\mathcal{P}_{i}$ and $\mathcal{P}$ to denote i.i.d.\ copies of $\mathcal{P}_{i}(u)$ and $\mathcal{P}(u)$, respectively. 

For $t\geq 0$ we set $\mathscr{T}_{t} := \{x \in \cU_\infty: \cB(x) \leq t\}$ and identify $\mathscr{T}_{t}$ as the \emph{genealogical tree} of individuals with birth time at most $t$. For a given choice of $X, W$, we say $(\mathscr{T}_{t})_{t \geq 0}$ is the genealogical tree process associated with an $(X,W)$-\emph{Crump--Mode--Jagers} branching process; often, we refer to $(\mathscr{T}_{t})_{t \geq 0}$ directly as an $(X,W)$-\emph{Crump--Mode--Jagers} branching process, viewed as a stochastic process in $t$.

With regards to the process $(\mathscr{T}_{t})_{t \geq 0}$, we define the stopping times $(\tau_{k})_{k \in \mathbb{N}_{0}}$ such that 
\be \label{eq:tauk}
\tau_{k} := \inf\{t \geq 0: |\mathscr{T}_{t}| \geq k\},
\ee 
where we adopt the convention that the infimum of the empty set is $\infty$. 
One readily verifies that $(|\mathscr{T}_{t}|)_{t \geq 0}$ is right-continuous, and thus $|\mathscr{T}_{\tau_{k}}| \geq k$ (when $\tau_k<\infty$). For each $k \in \mathbb{N}$ we define the tree $\mathcal{T}_{k}$ as the tree consisting of the first $k$ individuals in $\mathscr{T}_{\tau_{k}}$ ordered by birth time, breaking ties lexicographically. We call 
\be 
\tau_{\infty} := \lim_{k \to \infty} \tau_{k}
\ee 
the \emph{explosion time} of the process, at which the branching process has grown infinitely large (given the process survives). We also define the tree $\mathcal{T}_{\infty} := \bigcup_{k=1}^{\infty} \mathcal{T}_{k}$.

\begin{remark} \label{rem:notation-form}
	With the more commonly used notation for CMJ branching processes, one assigns a point process (denoted $\xi^{(u)}$) to each $u \in \cU_{\infty}$, and refers to the points $\sigma^{(u)}_{1} \leq  \sigma^{(u)}_{2}, \ldots$ associated with this point process (in the notation used here $\mathcal{B}(u1), \mathcal{B}(u2), \ldots$). We do not use this framework here, because this requires one to be able to write the measure $\xi^{(u)} = \sum_{i=1}^{\infty} \delta_{\sigma_{i}}$, which requires one to impose $\sigma$-finiteness assumptions on the point process (see, for example, \cite[Corollary~6.5]{last-penrose}). This $\sigma$-finiteness is implied by the classical Malthusian condition, but, in this general setting, we believe it is easier to have a framework where one can directly refer to the points $\mathcal{B}(u1), \mathcal{B}(u2), \ldots$ 
\end{remark}

\textbf{Notation. } Throughout the paper we use the following notation: We let $\N:=\{1,2,\ldots\}$ and set $\N_0:=\{0,1,\ldots\}$ and $\R_+:=[0,\infty)$. For $n\in\N$, we set $[n]:=\{1,\ldots, n\}$. For $x\in\R$, we let $\lceil x\rceil:=\inf\{n\in\Z: n\geq x\}$ and $\lfloor x\rfloor:=\sup\{n\in\Z: n\leq x\}$. For sequences $(a_n)_{n\in\N}$ and $(b_n)_{n\in\N}$ such that $b_n$ is positive for all $n$, we say that $a_n=o(b_n)$, $a_n=\mathcal{O}(b_n)$, and $a_n=\Theta(b_n)$ if $\lim_{n\to\infty} a_n/b_n=0$, if there exists a constant $C>0$ such that $|a_n|\leq Cb_n$ for all $n\in\N$, and if $a_n=\cO(b_n)$ and $b_n=\cO(a_n)$, respectively. For random variables $X,(X_n)_{n\in\N}$ we let $X_n\toindis X, X_n\toinp X$ and $X_n\toas X$ denote convergence in distribution, probability and almost sure convergence of $X_n$ to $X$, respectively. For a non-negative, real-valued random variable $X$ and $\lambda > 0$, we let
\[
\mathcal{M}_{\lambda}(X) := \E{\exp(\lambda X)} \quad \text{and} \quad \mathcal{L}_{\lambda}(X) := \E{\exp(-\lambda X)}.
\]
Finally, for real-valued random variables $X,Y$, we say  $X \preceq Y$ if $Y$ stochastically dominates $X$. 

\textbf{Structure of the paper. } We present the main results in Section~\ref{sec:results} and discuss applications in Section~\ref{sec:aplex}. We prove the main results and applications in Section~\ref{sec:proofmain}. Section~\ref{sec:examplesproofs} finally investigates a number of examples. 

\section{Results}\label{sec:results}

Throughout this paper, we make the following general assumptions, that are common when studying (explosive) $(X,W)$-CMJ processes. We assume that
the pairs $((X(uj))_{j \in \mathbb{N}}, W_{u})$ with $u\in\cU_\infty$ are i.i.d. For a given $w \in S$, we let $(X_{w}(uj))_{j \in \mathbb{N}}$ denote a sequence $(X(uj))_{j \in \mathbb{N}}$, conditionally on the weight $W_{u} = w$. We also assume throughout that, for any  $w \in S$ and any $u\in\cU_\infty$, the random variables 
$(X_w(ui))_{i\in\N}$ are mutually independent.

\subsection{Sufficient criteria for an infinite star in CMJ branching processes in random environment}\label{sec:star}
Our main interest and aim is to understand \emph{how} an explosive CMJ branching process explodes. We determine sufficient conditions under which the branching process explodes due to an individual giving birth to an infinite number of children in finite time (an \emph{infinite star}). The assumptions are as follows.
\begin{assumption} \label{ass:star}
	$\;$
	\begin{enumerate}
		\item\label{item:stardom} There exist positive random variables $(Y_{n})_{n \in \mathbb{N}_{0}}$ with finite mean such that, for any $w \in S$ and all $n\in\N_0$, 
		\begin{equation} \label{eq:stochastic-bound}
			\sum_{i=n+1}^{\infty} X_{w}(i) \preceq Y_{n}. 
		\end{equation}
		\item\label{item:starlimsup} There exists an increasing sequence $(\lambda_n)_{n\in\N}\subset(0,\infty)$ such that $\lim_{n\to\infty}\lambda_n=\infty$, and 
		\be 
		\sum_{n=1}^\infty \cM_{\lambda_n}(Y_n) \cL_{\lambda_n}(\cP_n(\varnothing))<\infty. 
		\ee 
		\item\label{item:starnonzero} We have $ \E{\sup\{k: X(1)=\cdots=X(k) = 0\}} < 1$. Additionally, for each $n \in \mathbb{N}$, almost surely 
		\be
		\sum_{i=n+1}^{\infty} X(i) > 0.
		\ee 
	\end{enumerate}
\end{assumption}

\begin{remark}
	Note that under Condition~\ref{item:stardom} of Assumption~\ref{ass:star}, we have $\P{ \tau_\infty<\infty} =1$, since, for example, the root $\varnothing$ produces an infinite number of children in finite time, almost surely. 
	\ensymboldremark
\end{remark}

\begin{remark} \label{rem:assumpt-motiv}
	In Assumption~\ref{ass:star} we can consider Condition~\ref{item:stardom} as a \emph{uniform explosivity} condition: regardless of its weight, the distribution of the time until an individual produces infinite offspring, after having already produced $n$ children, is dominated by $Y_n$. This assumption is essential to work around dependencies that arise due to the random environment (i.e.\ the weights of individuals).
	
	Condition~\ref{item:starlimsup} is used in the proof of Proposition~\ref{prop:local-explosions}, and provides an upper bound for the expected number of children that produce infinite offspring before their parent does.
	
	Condition~\ref{item:starnonzero} is a technical assumption, necessary to rule out certain trivial cases. Indeed, if, for example, $\E{\sup\{k:X(k)=0\}} > 1$, the tree consisting of all the individuals born \emph{instantaneously} at time $0$ is a supercritical Bienaym\'{e}-Galton-Watson branching process. Hence, with positive probability  this tree is infinitely large, whilst it may not contain an individual with infinite offspring.  \ensymboldremark
\end{remark}

\begin{remark}
	Condition~\ref{item:starlimsup} \emph{weakens and combines} two conditions in Assumption $2.2$ in~\cite{IyerLod23}. There, it was assumed, with $\lambda_n:=c\E{Y_n}^{-1}$ and $c<1$, that $\lambda_n$ increases and diverges, and that 
	\be 
	\limsup_{n\to\infty}\cM_{\lambda_n}(Y_n)<\infty \qquad\text{and}\qquad \ \sum_{n=1}^\infty \cL_{\lambda_n}(\cP_n(\varnothing))<\infty.
	\ee 
	It is clear that these assumptions imply Condition~\ref{item:starlimsup}.
	\ensymboldremark
\end{remark}

We then have our main result.

\begin{theorem}[Infinite star] \label{thm:star}
	Under Assumption~\ref{ass:star}, the infinite tree $\cT_{\infty}$ almost surely contains a node of infinite degree $($an infinite star$)$. 
\end{theorem}

\subsection{Application} 

We apply the result of Theorem~\ref{thm:star} when the inter-birth times are \emph{exponentially} distributed. In particular, this allows us to relate the results for $(X,W)$-CMJ branching processes to a family of recursively grown discrete trees, known as \emph{super-linear} preferential attachment trees with \emph{fitness}, introduced in~\cite{IyerLod23}. This is a sequence of trees where vertices are introduced one by one and are assigned vertex-weights (fitness values). When a new vertex is introduced, it connects to one of the vertices already in the tree, where vertices with high degree or high fitness are more likely to make connections with new vertices. 

We generally consider trees as being rooted with edges directed away from the root, and hence the number of `children' of a node corresponds to its \emph{out-degree}. More precisely, given a vertex labelled $v$ in a directed tree $T$ we let $\outdeg{v, T}$ denote its out-degree in $T$.  We now define the preferential attachment with fitness (PAF) model.

\begin{definition}\label{def:wrt}
	Suppose that $(W_{i})_{i \in \mathbb{N}}$ are i.i.d.\ copies of a random variable $W$ that takes values in $S$, and let $f\colon\mathbb{N}_{0}\times S \rightarrow (0, \infty)$ denote the \emph{fitness function}. A preferential attachment tree with fitness is the sequence of random trees $(T_{i})_{i \in \mathbb{N}}$ such that $T_{1}$ consists of a single node $1$ with weight $W_{1}$ and for $n \geq 2$, $T_{n}$ is constructed, conditionally on $T_{n-1}$, as follows:
	\begin{enumerate}
		\item Sample a vertex $j \in T_{n-1}$ with probability
		\begin{equation}
			\frac{f(\outdeg{j, T_{n-1}}, W_j)}{\sum_{i=1}^{n-1} f(\outdeg{i,T_{n-1}}, W_i)}.
		\end{equation}
		\item Connect $j$ with an edge directed outwards to a new vertex $n$ with weight $W_{n}$.
	\end{enumerate}
\end{definition}

The correspondence between preferential attachment trees with fitness and $(X,W)$-CMJ process with exponential inter-birth times is as follows. For each $u\in\cU_\infty$, set $X(ui)\sim \Exp{f(i-1,W_u)}$ for $i\in\N$. The trees $(\mathcal{T}_{i})_{i \in \N}$ associated with the $(X, W)$-CMJ process then satisfy 
\be 
\{T_i:i\in\N\}\overset d=\{\cT_i: i\in\N\}. 
\ee 
Additionally, with $T_\infty:=\cup_{n=1}^\infty T_n$, we also have $T_\infty \overset d=\cT_\infty$. As a result, the structural properties of $\cT_\infty$ can be translated to $T_\infty$. The equality in distribution is a consequence of the memory-less property and the fact that the minimum of exponential random variables is also exponentially distributed, with a rate given by the sum of the rates of the corresponding variables; see for example~\cite[Section~2.1]{rec-trees-fit}. The use of continuous-time embeddings of combinatorial processes is pioneered by Athreya and Karlin~\cite{Arthreya-karlin-embedding-68}.

We require the following assumptions on the fitness function $f$:
\be\label{eq:minfass}
\exists w^*\in S: \forall w\in S, j\in \N: f(j,w)\geq f(j,w^*) \quad \text{and } \quad \sum_{j=0}^{\infty} \frac{1}{f(j, w^{*})} < \infty.\tag{$w^*$}
\ee 
That is, there exists a minimiser $w^*\in S$ that, uniformly in $j\in \N$, minimises $f(j,\cdot)$, and the reciprocals of $f(j,w^*)$ are summable. Whilst the latter condition is necessary to have an $(X,W)$-CMJ branching process where individuals can produce an infinite number of children in finite time (i.e.\ where an infinite star can appear), the former assumption is to avoid certain technicalities only. We also define the quantities 
\be \label{eq:mu}
\mu_n^w:= \sum_{i=n}^\infty \frac{1}{f(i,w)}, \qquad w\in \R_+, n\in \N, \quad \text{and set } \mu_n:=\mu_n^{w^*}.
\ee 
Define $\mu_x$ for $x\in\R_+\setminus\N_0$ by linear interpolation. We then have the following corollary. 

\begin{corollary}[Infinite star in super-linear preferential attachment trees]\label{cor:suplinpa}
	Let $(T_{i})_{i \in \mathbb{N}}$ be a preferential attachment tree with fitness function $f$ that satisfies Assumption~\eqref{eq:minfass}. Let $(\lambda_n)_{n\in\N}\subset [0,\infty)$ be increasing and tend to infinity with $n$, such that there exists $N\in\N$ with $\lambda_n< f(i,w^*)$ for all $i\geq n\geq N$. If
	\be \label{eq:star-explosive-rif}
	\sum_{n=N}^{\infty} \bigg(\prod_{i=n}^\infty \frac{f(i,w^*)}{f(i,w^*)-\lambda_n}\bigg)\E{\prod_{i=0}^{n-1}\frac{f(i,W)}{f(i,W) + \lambda_n}} < \infty,
	\ee 
	then the tree $T_{\infty}$ contains a unique node of infinite degree, and no infinite path, almost surely. 
\end{corollary}

\begin{remark}\label{rem:otherdistr}
	Though we present a corollary for the specific choice of exponentially distributed inter-birth times, similar results can be proved to hold when considering other distributions for the inter-birth times (e.g.\ beta, gamma, Rayleigh distributions).\ensymboldremark
\end{remark}

The corollary follows by showing the conditions in Assumption~\ref{ass:star} are met. It is readily verified that Condition~\ref{item:starnonzero} is satisfied, as the exponential distribution does not have an atom at zero. Setting 
\be \label{eq:Ynexp}
Y_n:=\sum_{i=n+1}^\infty \wt X_{w^*}(i),
\ee 
where $\wt X_{w^*}(i)$ is an independent copy of $X_{w^*}(i)$, and applying Assumption~\eqref{eq:minfass},  Condition~\ref{item:stardom} follows. Finally, Condition~\ref{item:starlimsup} is equivalent to~\eqref{eq:star-explosive-rif} when the inter-birth times are exponentially distributed and $Y_n$ is as in~\eqref{eq:Ynexp}. The uniqueness of the node of infinite degree and absence of an infinite path follows from~\cite[Theorem $2.12$]{IyerLod23}, and is true in general for inter-birth time distributions that have no atoms.

We juxtapose Corollary~\ref{cor:suplinpa} with the following result from~\cite{IyerLod23}: 

\begin{theorem}[Theorem $3.4$ in~\cite{IyerLod23}] \label{thrm:path}
	Recall $\mu_n^w$ from~\eqref{eq:mu}. If, for some $c>1$ and all $w\geq 0$, we have 
	\be 
	\sum_{n=1}^\infty \E{\prod_{i=0}^\infty \frac{f(i,W)}{f(i,W)+c(\mu_n^w)^{-1}\log n}}=\infty, 
	\ee 
	then $T_\infty$ contains a unique infinite path and no node of infinite degree, almost surely. 
\end{theorem} 

In the upcoming sub-section we provide a range of examples for which the conditions in Corollary~\ref{cor:suplinpa} can be satisfied, for which we use $\lambda_n=\delta \mu_n^{-1} \log n$, where $\delta>0$ is a sufficiently small constant and $\mu_n$ is as in~\eqref{eq:mu}. For this choice of $\lambda_n$, we see that Corollary~\ref{cor:suplinpa} and Theorem~\ref{thrm:path} are close to being converse results. Indeed, we confirm this in a range of examples in the upcoming sub-section.

\subsubsection{Super-linear preferential attachment with fitness} 

We proceed by studying a family of preferential attachment trees (with fitness) for which we can verify the conditions provided in Corollary~\ref{cor:suplinpa}. In doing so, we extend the number of models for which it is known that $T_\infty$ contains a unique node of infinite degree from a handful of examples to larger, but more particular, class. This class consists of a family of \emph{multiplicative} fitness functions. That is, vertex-weights take values in $S=\R_+$, and $f(i,w)=g(w)s(i)$ for some functions $g\colon \R_+\to (0,\infty)$ and $s\colon \N_0\to (0,\infty)$. To satisfy Assumption~\eqref{eq:minfass}, we assume that
\be \label{eq:gsass}
\exists w^*\in \R_+\ \forall w\in \R_+: g(w)\geq g(w^*)>0,\qquad \text{and,}\qquad \ \sum_{i=0}^\infty \frac{1}{s(i)}<\infty. 
\ee 	
Our main interest is functions $s$ that grow \emph{barely faster than linear}, for which the summability condition in~\eqref{eq:gsass} is only just satisfied. For example, $s(n)=n\log(n+2)\log\log(n+3)^\sigma$, with $\sigma>1$. Earlier work by the author and Iyer in~\cite{IyerLod23} is only able to deal with cases for which $s(n)$ grows fast enough, that is, faster than $n(\log n)^\alpha$ with $\alpha>2$. Here, we extend this to a much wider range of functions that grow slower. 

We introduce the following assumptions on the function $s$. 

\begin{assumption}\label{ass:s}
	Suppose $s$ satisfies~\eqref{eq:gsass}. Moreover, suppose there exist $\beta\in(0,1)$ and $p\in(1,1+\beta)$ such that 
	\be 
	\lim_{n\to\infty}\frac{s(n)}{n^\beta}=\infty\qquad \text{and}\qquad \lim_{n\to\infty}\frac{s(n)}{n^p}=0,
	\ee 
	and there exist $C>0$ and $N\in\N$ such that for all $n\geq N$, 
	\be 
	s(n)\leq C\frac{n^{1+\beta-p}}{\log n}\inf_{i\geq n}s(i). 
	\ee 
\end{assumption}

\begin{remark}
	Though the first limit in Assumption~\ref{ass:s} may seem unnecessary when $s$ satisfies~\eqref{eq:gsass}, one can construct examples of $s$ with a sufficiently sparse summable subsequence that grows slowly. For example, with $\eps>0$ small, 
	\be 
	s(n)=\begin{cases}
		n^{(1+\eps)/2} &\mbox{if } \sqrt n\in\N, \\ 
		(n+2)(\log (n+2))^2 &\mbox{otherwise.}
	\end{cases}
	\ee 
	Though this choice of $s$ satisfies~\eqref{eq:gsass}, the subsequence $s(n_k)$ with $n_k=k^2$ grows relatively slowly compared to the values $s(n)$ when $n$ is not a perfect square, so that the first limit in Assumption~\ref{ass:s} is satisfied only for $\beta<\frac{1+\eps}{2}$. The inequality in Assumption~\ref{ass:s} quantifies how much slower such sparse slowly-growing sub-sequences are allowed to grow. Over-all, Assumption~\ref{ass:s} thus allows us to deal with some choices of $s$ with sparse slowly growing subsequences.\ensymboldremark
\end{remark}

We then have the following result.

\begin{theorem}[Barely super-linear preferential attachment with fitness]\label{thrm:suplinpafit}
	Suppose that $f(i,w)=g(w)s(i)$, where $s$ satisfies Assumption~\ref{ass:s} and $g$ and $s$ satisfy~\eqref{eq:gsass}. Suppose there exists a sequence $(k_n)_{n\in\N}$ and constants $\eps>0, \delta\in(0,\eps)$, and $n_0\in\N$,  such that both 
	\be 
	\P{g(W)>k_n}\leq n^{-(1+\eps)}, \qquad\text{for all }n\geq n_0,
	\ee 
	and
	\be \label{eq:assknmu}
	\lim_{n\to\infty} \frac{\mu_{(\delta\log(n)/(\mu_n k_n))^{1/\beta}}}{\mu_nk_n}=\infty
	\ee  
	are satisfied, where $\beta$ is as in Assumption~\ref{ass:s}. Then, the limiting infinite tree $T_\infty$ contains a unique infinite star and no infinite path, almost surely. 		  
\end{theorem}

\begin{remark}
	The conditions in Assumption~\ref{ass:s} allow us to construct a sequence $(\lambda_n)_{n\in\N}$ such that $\lambda_n<f(i,w)=g(w)s(i)$ for all $i\geq n$ and all $n$ large, whilst~\eqref{eq:assknmu} ensures that the summability condition in~\eqref{eq:star-explosive-rif} is met. \ensymboldremark
\end{remark}

\subsection{Examples} \label{sec:aplex}

To conclude the section, we present a (non-exhaustive) list of examples of functions $g$ and $s$ and vertex-weights distributions such that we can verify the conditions of Theorem~\ref{thrm:suplinpafit}. Without loss of generality, we can take $g(w)=w+1$, as one can change vertex-weight distribution accordingly for other choices of $g$. For each example, we have the following assumptions for the function $s$ and the tail distribution of the vertex-weights. With $\sigma>1$, $\kappa>0$, $\nu\in(0,1)$, $\gamma>1$, and $\alpha\in(1/2,1]$,
\be\ba 
({}&i) &&s(i)=(i+1)(\log(i+2))^\sigma, \quad  &&\P{W\geq x}=\Theta\big( \e^{-x^\kappa}\big),\\
({}&ii) && s(i)=(i+1)\log(i+2)\exp((\log\log(i+3))^\nu), \quad && \P{W\geq x}=\Theta( \exp\big(-\e^{(\log x)^\gamma}\big)\big),\\
({}&iii) && s(i)=(i+1)\log(i+2)(\log\log(i+3))^\sigma, \quad &&\P{W\geq x}=\Theta\big( \exp\big(-\e^{x^\kappa}\big)\big),\\
({}&iv) && s(i)=\begin{cases}
	i^\alpha &\mbox{if } \sqrt{i}\in\N,\\ 
	(i+1)(\log (i+2))^\sigma &\mbox{otherwise,} 
\end{cases} \quad && \P{W\geq x}=\Theta \big(\e^{-x^\kappa}\big).
\ea\ee 
We then have the following result. 

\begin{theorem}\label{thrm:examples}
	Consider the four examples $(i)$-$(iv)$. The infinite tree $T_\infty$ almost surely contains a unique infinite star and no infinite path when
	\be
	(i)\ \  (\sigma-1)\kappa>1;\qquad (ii)\ \  \nu\gamma>1;\qquad (iii)\ \ (\sigma-1)\kappa>1; \qquad (iv)\ \  (\sigma-1)\kappa>1.
	\ee
	When the above inequalities are reversed $($i.e.\ when changing the $>$ to a $<)$, almost surely $T_\infty$ contains a unique infinite path and no infinite star.
\end{theorem}

\begin{remark}
	For case $(i)$, the condition for the existence of an infinite path was already known~\cite{IyerLod23}. A sufficient condition for the existence of an infinite path was $(\sigma-1)\kappa>1+\tfrac1\kappa$, which we sharpen to $(\sigma-1)\kappa>1$ here to close the gap in the phase transition. For the other  cases, the conditions for the existence of both an infinite star and an infinite path are novel. For the existence of an infinite path we verify conditions from~\cite{IyerLod23}.  \ensymboldremark
\end{remark} 

\subsubsection{Super-linear preferential attachment without fitness }  To conclude, we consider the case $g\equiv 1$ so that $f(i,w)=s(i)$. That is, we consider a model where the evolution of the tree does not depend on the vertex-weights. Here, recent work of Iyer in~\cite{Iyer24} shows that the limiting tree $T_\infty$ almost surely contains a unique infinite star when 
\be \label{eq:Iyercond}
\sum_{i=0}^\infty \frac{1}{s(i)}<\infty\qquad\text{and}\qquad \exists \kappa >0\ \forall n\in\N_0: \max_{i\leq n}\frac{s(i)}{i+1}\leq \kappa \frac{s(n)}{n+1}.
\ee 
In the following result, we allow for functions $s$  that do not meet the second condition, thus extending the family of models for which we know a unique infinite star arises.

\begin{corollary}[Barely super-linear preferential attachment]\label{cor:bspa}
	Let $s$ satisfy~\eqref{eq:gsass} and Assumption~\ref{ass:s}, and suppose that 
	\be \label{eq:liminflb}
	\lim_{n\to\infty}\frac{\mu_{  (\log(n)\mu_n^{-1})^{1/\beta}}}{\mu_n}=\infty,
	\ee  
	where $\beta$ is as in Assumption~\ref{ass:s}. Then, the limiting infinite tree $T_\infty$ contains a unique vertex with infinite degree and no infinite path, almost surely.
\end{corollary}

\begin{remark}\label{rem:bspaproof}
	The condition in~\eqref{eq:liminflb} implies the result by applying Theorem~\ref{thrm:suplinpafit} to $g\equiv 1$.  \ensymboldremark
\end{remark}

Though not entirely general, the conditions in Corollary~\ref{cor:bspa} encompass a large class of functions $s$. The four cases in Theorem~\ref{thrm:examples} (ignoring the condition on the vertex-weights) satisfy the conditions, for example. In fact, we believe that the condition in~\eqref{eq:liminflb} is satisfied by \emph{any} function $s$ that satisfies Assumption~\ref{ass:s}. As we were unable to prove this, we included it as a condition in the Corollary (as well as in Theorem~\ref{thrm:suplinpafit}).

Corollary~\ref{cor:bspa} provides, to some extent, more general conditions under which $T_\infty$ contains a unique infinite star compared to~\eqref{eq:Iyercond}. Indeed, we can construct examples such that the second condition in~\eqref{eq:Iyercond} is not satisfied, but which do satisfy the weaker conditions in Corollary~\ref{cor:bspa} (see Case  (iv) in Section~\ref{sec:examplesproofs}). At the same time, the conditions in~\eqref{eq:Iyercond} can deal with a range of functions $s$ that are not within the range of Corollary~\ref{cor:bspa}.

\section{Proofs of main results}\label{sec:proofmain}

This section is dedicated to proving the results presented in Section~\ref{sec:results}. In Section~\ref{sec:mainthrmproof} we focus on the main result; Theorem~\ref{thm:star}. Section~\ref{sec:bspaproofs} is dedicated to proving Theorem~\ref{thrm:suplinpafit}. 

\subsection{Proof of Theorem~\ref{thm:star}}\label{sec:mainthrmproof} 

We introduce the following terminology, used in the remainder of the section. Recall $\cP_k$ and $\cP$ from~\eqref{eq:P}. For $a, b \in \cU_{\infty}$ we say that 
\[
\text{``$a$ has at least $k$ children before $b$ explodes''}\quad \text{if}\quad \ \mathcal{B}(a) + \mathcal{P}_{k}(a) < \mathcal{B}(b) + \mathcal{P}(b).
\]
We also say that 
\[
\text{``$a$ explodes before all of its ancestors''}\quad \text{if}\quad \ \forall \ell<|a|: \mathcal{B}(a) + \mathcal{P}(a) < \mathcal{B}(a_{|_\ell}) + \mathcal{P}(a_{|_\ell})
\]
Finally, for $a \in \cU_{\infty}$ with $|a| \geq 1$, we say that $a= a_{1} \cdots a_{m}$ is $a_{1}$-\emph{conservative} if, for each $j \in \left\{2, \ldots, m\right\}$, we have $a_{j} \leq a_{1}$. We then have the following result.

\begin{lemma} \label{lem:cons-bound-gen}
	Under Assumption~\ref{ass:star}, there exist $\zeta < 1$ and $K = K(\zeta) > 0$ such that for all $a_1 > K(\zeta)$ and all integers $m\in \mathbb{N}$,
	\be
	\sum_{\substack{a:|a| = m\\ a\text{ is $a_1$-conservative}}}\!\!\!\!\!\!\!\!\!\!\!\!\! \P{a \text{ has at least $a_1$ children before } \varnothing \text{ explodes}} \leq \zeta^{m-1} \cM_{\lambda_{a_1}}(Y_{a_1})\cL_{\lambda_{a_1}}(\mathcal{P}_{a_{1}}).
	\ee 
\end{lemma}

\begin{proof}
	Suppose that $a = a_1 \cdots a_m \in \cU_\infty$. For $a$ to have at least $a_1$ children before the explosion of $\varnothing$, all the ancestors of $a$ (including $a$) need to be born, and then $a$ needs to produce $a_1$ many children. All this needs to happen before $\varnothing$ produces infinitely many children, starting to count from its $(a_1+1)^{\mathrm{st}}$ child. That is, 
	\be  \label{eq:root-prob}
	\P{a \text{ has at least $a_1$ children before } \varnothing \text{ explodes}}
	= \mathbb P \bigg(\!\cP_{a_1}(a)+\!\sum_{j=2}^m\mathcal{P}_{a_j}(a_{|_{j-1}}) \leq\!\!\!\!\sum_{k= a_{1} + 1}^{\infty}\!\!\!\! X(k)\!\bigg).
	\ee 
	Using Assumption~\ref{ass:star}~\ref{item:stardom} and a Chernoff bound, with $\lambda_{a_1}>0$, we arrive at the upper bound
	\be 
	\mathbb P \bigg(\!\cP_{a_1}(a)+\!\sum_{j=2}^m\mathcal{P}_{a_j}(a_{|_{j-1}}) \leq Y_{a_1}\!\bigg)\leq \cM(\lambda_{a_1}(Y_{a_1}))\mathcal{L}_{\lambda_{a_1}}(\mathcal{P}_{a_1})\prod_{j=2}^m \mathcal{L}_{\lambda_{a_1}}(\mathcal{P}_{a_j}), 
	\ee
	where the last line follows from the fact that the sequence $(\mathcal{P}_{j}(u))_{j\in \mathbb{N}}$ is independent and distributed like $(\mathcal{P}_{j}(\varnothing))_{j\in \mathbb{N}}$ for any $u\in\cU_\infty$. When we sum over the possible conservative sequences $a$ that are $a_1$-conservative, each $a_{j}$ takes values between $1$ and $a_1$, for $j=2,\ldots, m$. Thus,
	\be\ba \label{eq:sumbound}
	\sum_{\substack{a:|a| = m\\ a\text{ $a_1$-conservative}}}\!\!\!\!\!\!\!\!\!\!\!{}&\P{a \text{ has at least $a_1$ children before } \varnothing \text{ explodes}}
	\\ \leq{}& \sum_{a_{2} = 1}^{a_{1}} \sum_{a_{3} = 1}^{a_{1}} \cdots \sum_{a_{m}=1}^{a_1}  \cM(\lambda_{a_1}(Y_{a_1}))\mathcal{L}_{\lambda_{a_1}}(\mathcal{P}_{a_1})\prod_{j=2}^m \mathcal{L}_{\lambda_{a_1}}(\mathcal{P}_{a_j})		\\ 
	= {}&  \cM(\lambda_{a_1}(Y_{a_1}))\mathcal{L}_{\lambda_{a_1}}(\mathcal{P}_{a_1}) \left(\sum_{n = 1}^{a_{1}} \mathcal{L}_{\lambda_{a_1}}(\mathcal{P}_{n}) \right)^{m-1}. 
	\ea\ee 
	It thus remains to show that, for $a_1$ sufficiently large,
	\[
	\sum_{n=1}^{a_1}\mathcal{L}_{\lambda_{a_1}}(\mathcal{P}_n) < \zeta. 
	\]
	As $\lambda_n>0$ for all $n\in\N$, we have $\cM_{\lambda_n}(Y_n)\geq 1$ for all $n\in\N$. Hence, by Assumption~\ref{ass:star}~\ref{item:starlimsup},
	\be 
	\sum_{n=1}^\infty \cL_{\lambda_n}(\cP_n)<\infty
	\ee 
	also holds.
	As a result, for any $\eta>0$ there exists $N = N(\eta)\in\N $ such that for all $a_1 > N$,
	\begin{equation} \label{eq:bound-part-1}
		\sum_{n=N}^{a_1} \mathcal{L}_{\lambda_{a_1}}(\mathcal{P}_n) < \sum_{n=N}^{\infty} \mathcal{L}_{\lambda_n}(\mathcal{P}_n) < \frac{\eta}{2},
	\end{equation}
	where the inequality uses the fact that $\lambda_n$ is increasing in $n$. On the other hand, since $\lambda_n$ diverges with $n$, bounded convergence (bounding the integrand by $1$) yields
	\be\label{eq:domconv}
	\lim_{a_1 \to \infty} \sum_{n = 1}^{N-1}\mathcal{L}_{\lambda_{a_1}}(\mathcal{P}_n) =\sum_{n=1}^{N-1} \P{\cP_n=0}=\sum_{n=1}^{N-1}\P{X(1)=\cdots =X(n)=0}.
	\ee
	We note that, by Assumption~\ref{ass:star}\ref{item:starnonzero}, there exists $\xi>0$ such that 
	\be 
	\E{\sup\{k:X(1)=\cdots= X(k)=0\}}=\sum_{n=1}^\infty \P{X(1)=\cdots =X(n)=0}<1-\xi.  
	\ee 
	Hence, the right-hand side of~\eqref{eq:domconv} is at most $1-\xi$ for any $N\in\N$. So, we first take $\eta<\xi$ and  $N$ large enough so that we have the upper bound in~\eqref{eq:bound-part-1}. Then, we take $K\geq N$ sufficiently large, so that for all $a_1\geq K$, 
	\begin{equation} \label{eq:bound-part-2}
		\sum_{n = 1}^{N-1} \mathcal{L}_{\lambda_{a_1}}(\mathcal{P}_{\ell}) < 1-\xi/2.
	\end{equation}
	Combined, we thus arrive at 
	\be 
	\sum_{n=1}^{a_1}\mathcal{L}_{\lambda_{a_1}}(\mathcal{P}_{\ell})<1-\xi/2=:\zeta.
	\ee 
	Using this in~\eqref{eq:sumbound}, we conclude the proof.  
\end{proof}

\noindent The above lemma provides an upper bound for the probability of the event that a vertex $a$ explodes before the root of $\cU_\infty$, in the case that $a$ is $a_1$-\emph{conservative}. When $a$ does not satisfy this condition, we can view $a$ as a concatenation of a number of conservative sequences. That is, we write $a=\overline b_1\cdots \overline b_\ell$, where $\overline b_i=b_{i,1}\ldots b_{i,m_i}$ for each $i\in[\ell]$ and for some $\ell\in\N, (m_i)_{i\in[\ell]}\in \N^\ell$, and $(b_{i,j})_{i\in[\ell],j\in[m_i]}$, such that $\overline b_i$ is $b_{i,1}$-conservative for each $i\in[\ell]$. By the independence of birth processes of distinct individuals (or in fact, the independence of disjoint sub-trees), we are able to apply Lemma~\ref{lem:cons-bound-gen} to each conservative sequence in the concatenation to arrive at a bound for the expected number of individuals that explode before all their ancestors.

\begin{proposition} \label{prop:local-explosions}
	Under Assumption~\ref{ass:star}, there exists $K'>0$ sufficiently large, such that 
	\be 
	\E{\left| \left\{a\in \cU_\infty: a_1>K', a\text{ explodes before all its ancestors} \right\} \right|}<\infty.
	\ee 
\end{proposition}

\begin{proof}
	As explained before the proposition statement, we think of sequences $a\in\cU_\infty$ as a concatenation of conservative sequences. Let $a=a_1\ldots a_m$ be a sequence of length $m\in\N$, and assume that there exist $k\in[m]$ and indices $1:=I_1<I_2<\ldots<I_k$ such that $a_1=:a_{I_1}<a_{I_2}<\ldots <a_{I_k}$ and $a_i\leq a_{I_j}$ for all $i\in \{I_j+1,\ldots, I_{j+1}-1\}$ and $j\in [k]$ (with $I_{k+1}:=m+1$). That is, the $I_j$ are the indices of the running maxima of $a$. We also set $I_{k+1} := m + 1$ and $a_{0} = \varnothing$. We think of $a$ as a concatenation of the conservative sequences $a_{I_j}\cdots a_{I_{j+1}-1}$, with $j\in[k]$. These sub-sequences can be seen as corresponding to an $a_{I_j}$-conservative individual, rooted at $a_{I_j-1}$, for each $j\in[k]$. We can thus apply Lemma~\ref{lem:cons-bound-gen} to  these sub-sequences. 
	
	Since, by definition we have $a_{I_{\ell +1}} >  a_{I_{\ell}}$, applying a similar logic to~\eqref{eq:root-prob}, we have the inclusion
	\begin{equation}
		\begin{aligned} \label{eq:ea}
			E_a:={}&\{a\text{ explodes}\text{ before any of its ancestors explodes}\}\\
			\subseteq{}&\bigcap_{\ell=1}^{k}\{a_1\cdots a_{I_{\ell+1}-1} \text{ gives birth to at least } a_{I_{\ell}}\text{ children before }a_1\cdots a_{I_{\ell}-1}\text{ explodes}\}
			\\  ={}& \bigcap_{\ell=1}^{k}\bigg\{ \mathcal{P}_{a_{I_{\ell}}}(a_{|_{I_{\ell+1}-1}})+\sum_{j=I_{\ell}}^{I_{\ell+1}-2}\mathcal{P}_{a_{j+1}}(a_{|_j}) \leq \sum_{i= a_{I_{\ell}} + 1}^{\infty} X(a_1 \cdots a_{I_{\ell}-1} i)\bigg\}
			=:\bigcap_{\ell=1}^k E_{a,\ell}.
		\end{aligned}
	\end{equation}
	Now, note that the events $(E_{a, \ell}, \ell \in [k])$ are not independent, since, for a given $\ell$, the term $\mathcal{P}_{a_{I_{\ell}}}(a_{|_{I_{\ell+1}-1}})$ in $E_{a, \ell}$ may be correlated with the summands $ X(a_1 \cdots a_{I_{\ell+1}-1} i)$ appearing in $E_{a, \ell+1}$ (via the weight $W_{a_{I_{\ell+1}-1}}$). However, as we assume that, for any $u\in \cU_\infty$, the random variables $(X(uj))_{j\in\N}$, conditionally on $W_u$, are independent, it follows that these events are conditionally independent, given the weights of $a$ and all its ancestors, $W_\varnothing, W_{a_1},W_{a_1a_2},\ldots, W_a$. Thus, 
	\begin{equation}
		\begin{aligned}
			\mathbb P{}&(E_a \, | \, W_\varnothing, W_{a_1},W_{a_1a_2},\ldots, W_a) 
			\\ & \leq \prod_{\ell=1}^{k} \mathbb P \bigg( \mathcal{P}_{a_{I_{\ell}}}(a_{|_{I_{\ell+1}-1}})+\sum_{j=I_{\ell}}^{I_{\ell+1}-2}\mathcal{P}_{a_{j+1}}(a_{|_j}) \leq \!\!\!\!\sum_{i= a_{I_{\ell}} + 1}^{\infty}\!\!\!\! X(a_1 \cdots a_{I_{\ell}-1} i) \, \bigg | \, W_\varnothing, W_{a_1},W_{a_1a_2},\ldots, W_a\bigg)
			\\ &  \stackrel{\eqref{eq:stochastic-bound}}{\leq} \prod_{\ell=1}^{k} \mathbb P\bigg( \mathcal{P}_{a_{I_{\ell}}}(a_{|_{I_{\ell+1}-1}})+\sum_{j=I_{\ell}}^{I_{\ell+1}-2}\mathcal{P}_{a_{j+1}}(a_{|_j}) \leq Y^{(a_1 \cdots a_{I_{\ell}-1})}_{a_{I_{\ell}}} \, \bigg | \, W_\varnothing, W_{a_1},W_{a_1a_2},\ldots, W_a\bigg)
			\\ &  = \prod_{\ell=1}^{k} \P{\wt E_{a, \ell} \, \Big| \, W_\varnothing, W_{a_1},W_{a_1a_2},\ldots, W_a}, 
		\end{aligned}
	\end{equation}
	where each $Y^{(a_1 \cdots a_{I_{\ell}-1})}_{a_{I_{\ell}}}$ is independent and distributed like $Y_{a_{I_{\ell}}}$ and does depend on the vertex-weights. Now, each of the events $ \wt E_{a, \ell}$ are independent, as they depend on different weights. Hence, so are each of the terms appearing in the product, so that taking expectations on both sides yields
	\begin{equation} \label{eq:probsplitineq}
		\P{E_a}\leq \prod_{\ell=1}^k \P{\wt E_{a,\ell}}.    
	\end{equation}
	We now let $d_j:=I_{j+1}-I_j-1$ for $j\in[k-1]$ and $d_k:=m-I_k$ denote the number of entries between the running maxima of $a$. We can then define, for $(d_j)_{j\in[k]}\in\N_0^k$ (with $[0]:=\emptyset$),
	\be \ba 
	\mathscr P_k({}&a_{I_1},a_{I_2},\ldots, a_{I_k}, d_1, \ldots, d_k)\\
	&:=\{a\in \cU_\infty: \text{ For all }j\in\{1,\ldots, k\}\text{ and all }i\in[d_j],\ a_{I_j+i}\in[a_{I_j}]\}
	\ea \ee 
	as the set of all sequences $a$ with running maxima $a_1=a_{I_1},\ldots, a_{I_k}$, and $d_j$ many entries between the $j^{\text{th}}$ and $(j+1)^{\text{th}}$ maximum. For brevity, we omit the arguments of $\mathscr P_k$. We then write the expected value in the proposition statement  as
	\be\ba \label{eq:splitsum1}
	\sum_{\substack{ a\in \cU_\infty \\ a_1>K'}}\P{E_a}=\sum_{m=1}^\infty \sum_{\substack{a: |a|=m\\ a_1>K'}}\P{E_a}\leq \sum_{m=1}^\infty \sum_{k=1}^m \sum_{a_{I_k}>\ldots >a_{I_1}>K'}\sum_{\substack{(d_\ell)_{\ell\in[k]}\in \N_0^k\\ \sum_{\ell=1}^k d_\ell=m-k}}\sum_{a\in \mathscr P_k}\prod_{\ell=1}^k\P{\wt E_{a,\ell}}.
	\ea\ee 
	In the first step, we introduce a sum over all sequence lengths $m$. In the second step, we furthermore sum over the number of running maxima $k$, the values of the running maxima $a_{I_1}, \ldots a_{I_k}$, the number of entries $d_\ell$ between each maxima $I_\ell$ and $I_{\ell+1}$ (or between $I_k$ and $m$ if $\ell=m$), and all sequences $a\in\mathscr P_k$ that admit such running maxima and inter-maxima lengths. Moreover, we use~\eqref{eq:probsplitineq} to bound $\P{E_a}$ from above, now that we know the number of running maxima in $a$.
	
	We can now take the sum over $a\in\mathscr P_k$ into the product, due to the fact that we can decompose each sequence  $a\in\mathscr P_k(a_{I_1},\ldots, a_{I_k},d_1,\ldots d_k)$ into a concatenation of sequences $a^{(1)}\ldots a^{(k)}$, with $a^{(\ell)}:=a_{I_\ell}\cdots a_{I_{\ell+1}-1}\in\mathscr P_1(a_{I_\ell},d_\ell)$ for each $\ell\in[k]$. This yields,
	\be 
	\sum_{a\in\mathscr P_k}\prod_{\ell=1}^k \P{\wt E_{a,\ell}}=\prod_{\ell=1}^k\!\!\!\! \sum_{\substack{a^{(\ell)}:|a^{(\ell)}|=d_\ell+1\\ \text{$a^{(\ell)}$ $a_{I_\ell}$-conservative}}}\!\!\!\!\!\!\!\!\!\!\!\P{\wt E_{a,\ell}}.
	\ee 
	We can then directly apply Lemma~\ref{lem:cons-bound-gen} to each of the sums in the product. Indeed, as $a_{I_\ell}>a_{I_1}=a_1>K'$, we can take $K'$ large enough to obtain, for some $\zeta<1$, the upper bound 
	\be 
	\prod_{\ell=1}^k \!\!\!\!\sum_{\substack{a^{(\ell)}:|a^{(\ell)}|=d_\ell+1\\ \text{$a^{(\ell)}$ $a_{I_\ell}$-conservative}}}\!\!\!\!\!\!\!\!\!\!\!\P{\wt E_{a,\ell}}\leq \prod_{\ell=1}^k \zeta^{d_\ell}\cM_{\lambda_{a_{I_\ell}}}(Y_{a_{I_\ell}}) \cL_{\lambda_{a_{I_\ell}}}(\mathcal{P}_{a_{I_\ell}}).
	\ee 
	We can take out the factors $\zeta^{d_\ell}$ and use that the $d_\ell$ sum to $m-k$. Using this in~\eqref{eq:splitsum1} yields
	\be\ba 
	\sum_{m=1}^\infty{}& \sum_{k=1}^m \sum_{a_{I_k}>\ldots >a_{I_1}>K'}\zeta^{m-k}\sum_{\substack{(d_j)_{j\in[k]}\in \N_0^k\\ \sum_{\ell=1}^k d_\ell=m-k}}\prod_{\ell=1}^k \cM_{\lambda_{a_{I_\ell}}}(Y_{a_{I_\ell}}) \cL_{\lambda_{a_{I_\ell}}}(\mathcal{P}_{a_{I_\ell}})\\ 
	&\leq \sum_{m=1}^\infty \sum_{k=1}^m \binom{m-1}{k-1}\zeta^{m-k}\sum_{a_{I_1}>K'}\cdots \sum_{a_{I_k}>K'}\prod_{\ell=1}^k \cM_{\lambda_{a_{I_\ell}}}(Y_{a_{I_\ell}}) \cL_{\lambda_{a_{I_\ell}}}(\mathcal{P}_{a_{I_\ell}})\\
	&=\sum_{m=1}^\infty \sum_{k=1}^m \binom{m-1}{k-1}\zeta^{m-k}\left(\sum_{n> K'} \cM_{\lambda_n}(Y_n)\cL_{\lambda_n}(\mathcal{P}_n)\right)^k.
	\ea\ee 
	By Assumption~\ref{ass:star}~\ref{item:starlimsup} we can bound the innermost sum from above by $\zeta/M$ for some $M>0$ when $K'$ is sufficiently large, so that we obtain the upper bound 
	\be 
	\sum_{m=1}^\infty \sum_{k=1}^m \binom{m-1}{k-1}\zeta^mM^{-k}=\frac{\zeta}{M}\sum_{m=1}^\infty  \Big(\zeta\big(1+\tfrac1M\big)\Big)^{m-1}<\infty, 
	\ee 
	where the last step follows when $M$ is large enough such that $\zeta(1+1/M)<1$, which holds by choosing $K'$ sufficiently large, as follows from~\eqref{eq:bound-part-1}.
\end{proof}

We then state three results,  proved by the author and Iyer in~\cite{IyerLod23}. First, we introduce for $L\in \N$,
\[
\mathcal{U}_{L} := \left\{u \in \mathcal{U}_{\infty}: u_{i} \leq L \text{ for all } i \in [|u|] \right\} \cup \{\varnothing\}.
\]
Following~\cite{Oliveira-spencer}, we say elements of $\mathcal{U}_{L}$ are $L$-\emph{moderate}. We then have the following result, which tells us that the sub-tree consisting of $L$-moderate individuals does not explode, almost surely.

\begin{proposition}[\cite{IyerLod23}] \label{prop:finite-l-moderate}
	Let $(\mathscr{T}_{t})_{t \geq 0}$ be a $(X,W)$-CMJ branching process satisfying Condition~\ref{item:starnonzero} of Assumption~\ref{ass:star}.  Then, almost surely, for all $t \in (0,\infty)$, we have $|\mathscr{T}_{t}\cap \cU_L| < \infty$ almost surely.
\end{proposition}

\noindent Recall that, for an $(X,W)$-Crump--Mode--Jagers branching process $(\mathscr{T}_{t})_{t \geq 0}$, we have  
\[\tau_{\infty} = \lim_{k \to \infty} \tau_{k} = \inf\left\{t > 0: |\mathscr{T}_t| = \infty\right\}.\]
Recall also that we have $\mathcal{T}_{\infty} = \bigcup_{k=1}^{\infty} \mathcal{T}_{k} = \bigcup_{k=1}^{\infty} \mathscr{T}_{\tau_{k}}$.  We then have the following result. 
\begin{lemma}[\cite{IyerLod23}] \label{lemma:tree-ident}
	Let $(\mathscr{T}_{t})_{t \geq 0}$ be  a $(X,W)$-CMJ branching process satisfying Condition~\ref{item:starnonzero} of Assumption~\ref{ass:star}. Then, almost surely, $\mathcal{T}_{\infty} = \left\{u \in \mathcal{U}_{\infty}: \mathcal{B}(u) < \tau_{\infty}\right\} \subseteq \mathscr{T}_{\tau_{\infty}}$.
\end{lemma}

Finally, we relate $\tau_\infty$ with the explosion times of all individuals in $\cU_\infty$. 

\begin{lemma}[\cite{IyerLod23}]\label{lem:expl-min-exp-time}
	Let $(\mathscr{T}_{t})_{t \geq 0}$ be  a $(X,W)$-CMJ branching process that satisfies Assumption~\ref{ass:star}. Almost surely, $\tau_{\infty} = \inf_{u \in \cU_{\infty}}\{\mathcal{B}(u) + \mathcal{P}(u)\}$. 
\end{lemma}

Combining these three lemmas with Proposition~\ref{prop:local-explosions}, we are ready to prove Theorem~\ref{thm:star}. Though the proof is identical to the proof of Theorem $2.5$ in~\cite{IyerLod23}, as it uses the above three results and a weaker version of Proposition~\ref{prop:local-explosions}, we include it for the article to be self-contained.

\begin{proof}[Proof of Theorem~\ref{thm:star}]
	Fix $K'$ as in Proposition~\ref{prop:local-explosions}. We may view any $w \in \mathcal{U}_{\infty}$ as a concatenation $w = uv$, where $u \in \mathcal{U}_{K'}$ is $K'$-moderate, and $v = v_1 \cdots v_{k}$, where $v_{1} > K'$ (here we also allow $v$ to be empty, so that $K'$-moderate nodes $w$ may also be interpreted as a concatenation). Now, note that (on $\mathcal{B}(u) < \infty$) the birth times $\mathcal{B}(uv) - \mathcal{B}(u) \sim \mathcal{B}(v)$, and thus, by arguments analogous to those appearing in Proposition~\ref{prop:local-explosions}, for any $u \in \mathcal{U}_{\infty}$ (in particular for $u \in \mathcal{U}_{K'}$),
	\be \label{eq:moderate-prefix}
	\E{\left| \left\{a = a_1 \cdots a_{m} \in \cU_\infty: a_1>K', u a\text{ explodes before } ua_{|_{m-1}}, ua_{|_{m-2}}, \ldots, u\right\}\right| }<\infty.
	\ee 
	Now, since $\tau_{\infty} < \infty$ almost surely, we infer from Proposition~\ref{prop:finite-l-moderate} with $L = K'$, that $|\{u \in \mathcal{U}_{K'}: \mathcal{B}(u) \leq \tau_{\infty}\}| < \infty$ almost surely. Therefore, by~\eqref{eq:moderate-prefix}, the set 
	\be \label{eq:Sfin} 
	E_{\mathrm{expl}}:=\left\{u \in \cU_\infty: \mathcal{B}(u) \leq \tau_{\infty}, u \text{ explodes before all of its ancestors} \right\}.
	\ee
	is finite almost surely.
	By the definition of $E_{\mathrm{expl}}$ and the fact that the infimum of a finite set is attained by (at least) one of its elements, Lemma~\ref{lem:expl-min-exp-time} implies that, almost surely 
	\[
	\exists u^{*}\in E_{\mathrm{expl}}: \ \mathcal{B}(u^{*}) + \mathcal{P}(u^{*}) = \inf_{v \in E_{\mathrm{expl}}} \{\mathcal{B}(v) + \mathcal{P}(v)\} = \inf_{u \in \cU_\infty} \{\mathcal{B}(u) + \mathcal{P}(u)\} = \tau_{\infty}.
	\]
	This implies that $u^{*}$ has infinite degree in $\mathscr{T}_{\tau_{\infty}}$. By Condition~\ref{item:starnonzero} of Assumption~\ref{ass:star}, we have $\sum_{i=n+1}^{\infty} X(u^{*}i) > 0$ almost surely, so that $\mathcal{B}(u^{*}i) < \tau_{\infty}$ for each $i\in\N$, almost surely. Therefore, by Lemma~\ref{lemma:tree-ident}, $u^{*}$ has infinite degree in $\mathcal{T}_{\infty}$ as well, which yields the desired result. 
\end{proof}

\subsection{Proof of Theorem~\ref{thrm:suplinpafit}}\label{sec:bspaproofs}

\begin{proof}[Proof of Theorem~\ref{thrm:suplinpafit}]
	The desired result follows by verifying the conditions on $ \lambda_n$ and the vertex-weights in Corollary~\ref{cor:suplinpa}.  We start by constructing a sequence $(\lambda_n)_{n\in\N}$ such that there exists $N\in\N$ for which we have $\lambda_n<f(i,w)$ for all $i\geq n\geq N$. We recall that there exist $\eps>0$, $n_0\in\N$, and a sequence $(k_n)_{n\in\N}$ such that $\P{g(W)>k_n}\leq n^{-(1+\eps)}$ for all $n\geq n_0$. Then, let $\delta\in(0,\eps)$ as in the statement of Theorem~\ref{thrm:suplinpafit}, recall $\mu_n$ from~\eqref{eq:mu}, and set 
	\be\label{eq:lambda}
	\lambda_n:=\delta\log(n)\mu_n^{-1}, 
	\ee  
	It is clear that $\mu_n$ is decreasing in $n$ and tends to zero by~\eqref{eq:gsass}, so that $\lambda_n$ is increasing and diverges. Then, by Assumption~\ref{ass:s}, we fix $\eta>0$ small and take $N$ large so that for all $n\geq N$,
	\be \label{eq:mubound}
	\mu_n^{-1}=\bigg(\sum_{j=n+1}^\infty \frac{1}{g(w^*)s(j)}\bigg)^{-1}\leq g(w^*)\bigg(\sum_{j=n+1}^\infty \frac{1}{\eta j^p}\bigg)^{-1}=(\eta g(w^*)(p-1)+o(1)) n^{p-1}.
	\ee 
	At the same time, again for $N$ large and all $n\geq N$, we have $s(n)\geq \eta^{-1} n^\beta$, so that 
	\be 
	\lambda_n=\delta \log(n)\mu_n^{-1}\leq (\eta \delta g(w^*)(p-1)+o(1))\log(n) n^{p-1}\leq (\eta^2 \delta g(w^*)(p-1)+o(1))\frac{\log(n)}{n^{(1+\beta)-p}} s(n).
	\ee 
	The final condition of Assumption~\ref{ass:s} then implies that 
	\be 
	\lambda_n\leq (\eta^2 \delta C g(w^*)(p-1)+o(1)) \inf_{i\geq n}s(i), 
	\ee 
	for all $n\geq N$ and some large $N\in\N$. As we can choose $\eta$ arbitrarily small, we can make the constant on the right-hand side arbitrarily small, so that 
	\be
	\lambda_n=o(\inf_{i\geq n}s(i)),
	\ee 
	as desired. 
	
	We then prove~\eqref{eq:star-explosive-rif}. \invisible{ 
		We recall that there exist $\eps>0$, $n_0\in\N$, and a sequence $(k_n)_{n\in\N}$ such that $\P{g(W)>k_n}\leq n^{-(1+\eps)}$ for all $n\geq n_0$. Then, we fix $\delta\in(0,\eps)$, recall $\mu_n$ from~\eqref{eq:mu}, and set 
		\be\label{eq:lambda}
		\lambda_n:=\delta\log(n)\mu_n^{-1}, 
		\ee  
		It is clear that $\mu_n$ is decreasing in $n$ and tends to zero by~\eqref{eq:gsass}, so that $\lambda_n$ is increasing and diverges. Furthermore, as $s$ is regularly varying with exponent one, it follows that $\mu_n$, and hence also $\lambda_n$, is slowly varying in $n$. Hence, since $n=o(s(n))$, it follows that $\lambda_n=o(s(n))$. Furthermore, as $s$ is regularly varying with exponent one, we have $\inf_{i\geq n}s(i)=(1+o(1))s(n)$, so that 
		\be 
		\inf_{i\geq n}f(i,w^*)=g(w^*)\inf_{i\geq n}s(i)\gg \lambda_n, 
		\ee 
		so that there exists $N\in\N$ such that the inequality $\lambda_n<f(i,w^*)$ for all $i\geq n$ and all $n\geq N $ is satisfied.} 
	By the choice of $f$ and assuming (without loss of generality) that $g(w^*)=1$, we have that~\eqref{eq:star-explosive-rif} is equivalent to
	\be \label{eq:condbequiv}
	\sum_{n=N}^\infty \prod_{i=n}^\infty \bigg(\frac{s(i)}{s(i)-\lambda_n}\bigg)\E{\prod_{i=0}^{n-1}\frac{g(W)s(i)}{g(W)s(i)+\lambda_n}}<\infty.
	\ee 
	By using that $1-x\leq \e^{-x}$ for all $x\in \R$, we can bound each term in the sum from above by 
	\be\ba   \label{eq:condbfitness}
	\exp{}&\bigg(\lambda_n \sum_{i=n}^\infty \frac{1}{s(i)-\lambda_n}-\lambda_n \sum_{i=0}^{n-1} \frac{1}{k_ns(i)+\lambda_n}\bigg)+\exp\bigg(\lambda_n \sum_{i=n}^\infty \frac{1}{s(i)-\lambda_n}\bigg)\P{g(W)>k_n}\\ 
	\leq{}& \exp\bigg(\lambda_n\sum_{i=n}^\infty \frac{1}{s(i)-\lambda_n}- \frac{\lambda_n}{k_n}\sum_{i=0}^{n-1}\frac{1}{s(i)+\lambda_n /k_n}\bigg)+\exp\bigg(\lambda_n\sum_{i=n}^\infty\frac{1}{s(i)-\lambda_n}\bigg)\frac{1}{n^{1+\eps}},
	\ea \ee 
	where we take $N\geq n_0$ and use that $\P{g(W)>k_n}\leq n^{-(1+\eps)}$ for all $n\geq N$ to arrive at the upper bound. Using that $\lambda_n=o(\inf_{i\geq n}s(i))$ and $g(w^*)=1$,
	\be 
	\sum_{i=n}^\infty \frac{1}{s(i)-\lambda_n}=(1+o(1))\mu_n. 
	\ee 
	Then, by the choice of $\lambda_n$ in~\eqref{eq:lambda}, we can write~\eqref{eq:condbfitness} as 
	\be \label{eq:condbfitness2}
	\exp\bigg((\delta+o(1))\log(n)- \frac{\lambda_n}{k_n}\sum_{i=0}^{n-1}\frac{1}{s(i)+\lambda_n /k_n}\bigg)+n^{-(1+\eps-\delta+o(1))}.
	\ee 
	Since $\delta<\eps$, the second term is summable in $n$. It thus remains to bound the first term.
	
	\invisible{We recall that $\sum_{j=1}^\infty 1/s(j)<\infty$, so that  $m=o(\inf_{j\geq m}s(j))$ as $m$ tends to infinity. Indeed, $\inf_{j\geq m}s(j)=(1+o(1))s(m)\gg m$ due to the summability condition on $s$.}  We recall that $p\in(1,1+\beta)$ and that $k_n\geq g(w^*)=1$ for all $n$. As a result, for some constant $C_\beta>0$ and using~\eqref{eq:mubound}, 
	\be 
	\Big(\frac{\lambda_n}{k_n}\Big)^{1/\beta}\leq   \lambda_n^{1/\beta}\leq  C_\beta (\log n)^{1/\beta} n^{(p-1)/\beta}=o(n). 
	\ee 
	We then observe that $\lambda_n/k_n$ tends to infinity with $n$. Indeed, as $\sup_{x\geq 0}\mu_x$ is finite (where we define $\mu_x$ for non-integer $x$ by linear interpolation), for all $n$ large,
	\be 
	\frac{\lambda_n}{k_n}=\delta \frac{\log n}{\mu_n k_n}\geq \frac{\mu_{(\delta \log (n)/(\mu_nk_n))^{1/\beta}}}{\mu_nk_n}, 
	\ee 
	and the right-hand side tends to infinity with $n$ by the condition in~\eqref{eq:liminflb}. Further, Assumption~\ref{ass:s} yields that $s(i)\geq M i^\beta$ for any $M$ large and all $i\geq I=I(M)\in\N$. Combined, we obtain for all $n$ large that 
	\be 
	s(i)\geq Mi^\beta \geq M\frac{\lambda_n}{k_n}, \qquad \text{for all } i\geq \Big(\frac{\lambda_n}{k_n}\Big)^{1/\beta}.
	\ee 
	As a result, 
	\be 
	\sum_{i=0}^{n-1}\frac{1}{s(i)+\lambda_n /k_n}\geq \sum_{i=(\lambda_n/k_n)^{1/\beta}}^{n-1}\frac{1}{s(i)+\lambda_n/k_n}=\frac{\mu_{(\lambda_n/k_n)^{1/\beta}}-\mu_n}{1+1/M}.
	\ee 
	Since $\lambda_n=\delta \log(n)\mu_n^{-1}$, we use the above in the first term in~\eqref{eq:condbfitness2} to obtain
	\be\label{eq:deltabound} 
	\exp\Big(\Big(\delta\frac{2M+1}{M+1}+o(1)\Big)\log(n)- \delta\frac{M}{M+1}\log(n)\frac{\mu_{(\lambda_n/k_n)^{1/\beta}}}{\mu_nk_n}\Big).
	\ee 
	By the assumption in~\eqref{eq:assknmu} it follows that we can bound the entire term from above by $n^{-C}$ for any $C>0$, since the ratio $\mu_{(\lambda_n/k_n)^{1/\beta}}/(\mu_nk_n)$ can be bounded from below by a sufficiently large constant for all $n$ large by~\eqref{eq:liminflb}. This yields~\eqref{eq:condbequiv} and concludes the proof.
\end{proof}

\section{Examples}\label{sec:examplesproofs}

In this section, we discuss the examples that Theorem~\ref{thrm:examples} deals with. The proof comes down to verifying the conditions in Theorem~\ref{thrm:suplinpafit}. We recall that we set $g(w)=w+1$. In each case, we asymptotically determine $\mu_n$, provide $k_n$, and show that the required assumptions are satisfied.

\begin{proof}[Proof of Theorem~\ref{thrm:examples}]
	\textbf{Infinite star. } To prove there exists a unique infinite star almost surely, we verify the conditions in Theorem~\ref{thrm:suplinpafit}. 
	
	\textbf{Case (i) } It is clear that Assumption~\ref{ass:s} is satisfied with $\beta=1$ and any $p\in(1,2)$. By switching from summation to integration,
	\be \label{eq:mui}
	\mu_n=\sum_{i=n}^\infty \frac{1}{s(i)}=\int_n^\infty\frac{1+o(1)}{x(\log x)^\sigma}\,\dd x=\int_{\log n}^\infty \frac{1+o(1)}{y^\sigma}\, \dd y=\big((\sigma-1)^{-1}+o(1)\big)(\log n)^{-(\sigma-1)}.  
	\ee 
	We set $k_n:=((1+\eps)\log n)^{1/\kappa}$ with $\eps>0$, and take $\delta\in(0,\eps)$. There exists $C_1>0$ such that 
	\be \label{eq:Wtaili}
	\P{g(W)>k_n}\leq C_1\e^{-k_n^\kappa}=C_1n^{-(1+\eps)},
	\ee 
	which is summable. Then, there exists $C_2>0$ such that
	\be 
	\frac{\mu_{\delta \log(n)/(\mu_nk_n)}}{\mu_nk_n}=(C_2+o(1))\frac{(\log n)^{\sigma-1}(\log\log n)^{-(\sigma-1)}}{(\log n)^{1/\kappa}}=(\log n)^{(\sigma-1)-1/\kappa-o(1)}.
	\ee 
	When $(\sigma-1)\kappa>1$, this quantity diverges with $n$, so that the condition in Theorem~\ref{thrm:suplinpafit} holds.

	\textbf{Case (ii) } It is clear that Assumption~\ref{ass:s} is satisfied with $\beta=1$ and any $p\in(1,2)$. By switching from summation to integration and substituting $y=(\log\log x)^\nu$, 
	\be \label{eq:muii}
	\mu_n=\sum_{i=n}^\infty \frac{1}{s(i)}=\int_n^\infty \!\!\!\frac{1+o(1)}{x\log x\exp((\log\log x)^\nu)}\,\dd x=(1+o(1))\int_{(\log\log n)^\nu}^\infty \!\!\!\!\!\!\!\!\!y^{1-1/\nu}\e^{-y}\dd y.
	\ee 
	By using the incomplete gamma function, we arrive at $\mu_n=\exp(-(1+o(1))(\log\log n)^\nu)$. We set $k_n:=\exp((\log((1+\eps)\log n))^{1/\gamma})$ with $\eps>0$, and take $\delta\in(0,\eps)$. There exists $C_1>0$ such that
	\be  \label{eq:Wtailii}
	\P{g(W)>k_n}\leq C_1\exp\big(-\e^{(\log k_n)^\gamma}\big)=C_1n^{-(1+\eps)}, 
	\ee 
	which is summable. Then,
	\be 
	\frac{\mu_{\delta \log(n)/(\mu_nk_n)}}{\mu_nk_n}=\exp\Big(\big[(\log\log n)^\nu-(\log\log n)^{1/\gamma}-\big(1-\tfrac1\gamma\big)^\nu (\log\log\log n)^\nu\big](1+o(1))\Big).
	\ee 
	When $\nu\gamma>1$, this quantity diverges with $n$, so that the condition in Theorem~\ref{thrm:suplinpafit} holds.
	
	\textbf{Case (iii) } It is clear that Assumption~\ref{ass:s} is satisfied with $\beta=1$ and any $p\in(1,2)$. By switching from summation to integration and with similar steps as in~\eqref{eq:mui},
	\be \label{eq:muiii}
	\mu_n=\sum_{i=n}^\infty \frac{1}{s(i)}=\int_n^\infty \frac{1+o(1)}{x\log(x)(\log\log x)^\sigma}\,\dd x=\big((\sigma-1)^{-1}+o(1)\big)(\log\log n)^{-(\sigma-1)}.
	\ee 
	We set $k_n:=(\log((1+\eps)\log n))^{1/\kappa}$ with $\eps>0$, and take $\delta\in(0,\eps)$. There exists $C_1>0$ such that
	\be  \label{eq:Wtailiii}
	\P{g(W)>k_n}\leq C_1\exp\big(-\e^{k_n^\kappa}\big)=C_1n^{-(1+\eps)}, 
	\ee 
	which is summable. Then, there exists $C_2>0$ such that
	\be 
	\frac{\mu_{\delta \log(n)/(\mu_nk_n)}}{\mu_nk_n}=(C_2+o(1))\frac{(\log \log n)^{\sigma-1}(\log\log\log  n)^{-(\sigma-1)}}{(\log\log  n)^{1/\kappa}}=(\log \log n)^{(\sigma-1)-1/\kappa-o(1)}.
	\ee 
	When $(\sigma-1)\kappa>1$, this quantity diverges with $n$, so that the condition in Theorem~\ref{thrm:suplinpafit} holds.
	
	\textbf{Case (iv) } It is clear that the first part of Assumption~\ref{ass:s} is satisfied for any  $\beta<\alpha$ and any $p\in(1,1+\beta)$. Then, as for large $n$ the smallest value $s(i)$ among all $i\geq n$ is attained at $i=\lceil \sqrt n\rceil^2$, with value $s(i)=\lceil \sqrt n\rceil^{2\alpha}=(1+o(1))n^\alpha$, it follows that the second part of Assumption~\ref{ass:s} is also satisfied, since $\alpha\in(\tfrac12,1]$ and  we can thus choose $p$ close enough to $1$ and $\beta$ close enough to $\alpha$ so that $1+\beta-p>\tfrac12>1-\alpha$. Hence, for all $n$ large, 
	\be 
	s(n)\leq (n+1)(\log(n+2))^\sigma \leq  \frac{n^{1+\beta-p}}{\log n} n^\alpha=(1+o(1))\frac{n^{1+\beta-p}}{\log n}  \inf_{i\geq n}s(i).
	\ee 
	The last part of Assumption~\eqref{ass:s} is thus satisfied with $C>1$ and $N$  sufficiently large.  Then,
	\be \label{eq:muiv}
	\mu_n =\sum_{i=n}^\infty \frac{1}{s(i)}=\sum_{\substack{i=n\\ \sqrt i\in\N}}^\infty \frac{1}{i^\alpha}+\sum_{\substack{i=n\\ \sqrt i\not\in\N}}^\infty \frac{1}{(i+1)(\log(i+2))^\sigma}. 
	\ee
	As $i^\alpha\leq (i+1)(\log(i+2))^\sigma$ for all $i\geq n$ when $n$ is large since $\alpha\leq 1$, we can use~\eqref{eq:mui} to obtain the lower bound 
	\be 
	\mu_n\geq \big((\sigma-1)^{-1}+o(1)\big)(\log n)^{-(\sigma-1)}. 
	\ee 
	For an upper bound, we include all $i\geq n$ in the second sum in~\eqref{eq:muiv} and again use~\eqref{eq:mui}. The first sum on the right-hand side of~\eqref{eq:muiv} is then bounded from above by
	\be 
	\sum_{\substack{i=n\\ \sqrt i\in\N}}\frac{1}{i^\alpha}=\sum_{i=\lceil \sqrt n\rceil}^\infty i^{-2\alpha}= \frac{1+o(1)}{2\alpha -1}n^{1/2-\alpha}.
	\ee 
	As a result, the upper bound matches the lower bound regardless of the value of $\alpha\in(\tfrac12,1]$, to obtain 
	\be \label{eq:muivreal}
	\mu_n=\big((\sigma-1)^{-1}+o(1)\big)(\log n)^{-(\sigma-1)}.
	\ee 
	The remainder of the computations then follow the same approach as Case (i). 
	
	\textbf{Infinite path. } To prove there exists a unique infinite path, we use~\cite[Lemma $6.4$]{IyerLod23} to verify the conditions in Theorem~\ref{thrm:path}. That is, if, for any $w\geq 0$, 
	\be 
	\limsup_{n\to\infty}\frac{1}{\mu_n^w}\sum_{i=0}^\infty \frac{1}{f(i,k_n)}<1,
	\ee 
	where $k_n$ is a sequence such that $\P{g(W)>k_n}$ is not summable, then $T_\infty$ contains a unique infinite path and no node of infinite degree, almost surely. As $f(i,x)=g(x)s(i)=(x+1)s(i)$, it follows from~\eqref{eq:gsass} that is suffices to show that $\mu_n k_n$ diverges.
	
	\textbf{Case (i) } With $\mu_n$ as in~\eqref{eq:mui} and $k_n:=((1-\eps)\log n)^{1/\kappa}$, it follows that $\P{g(W)>k_n}$ is not summable when the inequality in~\eqref{eq:Wtaili} is reversed (and using a constant $C_1'<C_1$). We then conclude that $\mu_nk_n$ diverges when $(\sigma-1)\kappa<1$.
	
	\textbf{Case (ii) } With $\mu_n$ as in~\eqref{eq:muii} and $k_n:=\exp((\log(1-\eps)\log n))^{1/\gamma})$, it follows that $\P{g(W)>k_n}$ is not summable when the inequality in~\eqref{eq:Wtailii} is reversed (and using a constant $C_1'<C_1$). We then conclude that $\mu_nk_n$ diverges when $\nu \gamma <1$.
	
	\textbf{Case (i) } With $\mu_n$ as in~\eqref{eq:muiii} and $k_n:=(\log((1-\eps)\log n))^{1/\kappa}$, it follows that $\P{g(W)>k_n}$ is not summable when the inequality in~\eqref{eq:Wtailiii} is reversed (and using a constant $C_1'<C_1$). We then conclude that $\mu_nk_n$ diverges when $(\sigma-1)\kappa<1$.
	
	\textbf{Case (iv) } With $\mu_n$ as in~\eqref{eq:muivreal}, we can follow the same steps as in Case (i) to arrive at the desired result.
\end{proof} 
	
	\section*{Acknowledgements}
	
	\noindent The author wishes to thank Thomas Gottfried for some useful discussions, as well as Tejas Iyer for the stimulating collaboration on work that precedes the present paper and useful discussions. 
	
	The author has received funding from the European Union’s Horizon 2022 research and innovation programme under the Marie Sk\l{}odowska-Curie grant agreement no.\ $101108569$.

	\bibliographystyle{abbrv}
	\bibliography{refs}

\end{document}